\documentclass{amsart}
\usepackage{amsmath}
\usepackage{amsthm}
\usepackage{paralist}
\usepackage{geometry}
\usepackage{fullpage}

\usepackage{listings}
\usepackage{xcolor,soul}
\usepackage{tikz-cd}
\usetikzlibrary{decorations.markings}
\tikzset{negated/.style={
        decoration={markings,
            mark= at position 0.5 with {
                \node[transform shape] (tempnode) {$\backslash$};
            }
        },
        postaction={decorate}
    }
}
\usepackage{mathtools}
\usepackage{extpfeil}

\newtheorem{theorem}{Theorem}[section]
\newtheorem{proposition}[theorem]{Proposition}
\newtheorem{lemma}[theorem]{Lemma}

\usepackage{enumitem}  

\setenumerate{label={\normalfont(\textit{\roman*})}, ref={(\textrm{\roman*})},
wide,
}

\usepackage{ifthen}
\usepackage{scalerel}
\usepackage{hyperref}
\usepackage{mathrsfs}
\usepackage[mathscr]{euscript}
\usepackage{lineno}

\usepackage{stmaryrd}

\newtheorem{alphatheorem}{Theorem}

\theoremstyle{definition}

\newtheorem{remark}[theorem]{Remark}
\newtheorem*{remark*}{Remark}

\newtheorem{question}[theorem]{Question}

\renewcommand{\Re}{\operatorname{Re}}
\renewcommand{\Im}{\operatorname{Im}}

\newcommand{\supp}%
{\operatorname{supp}}

\newcommand{\fp}[1]{\left\{ #1 \right\} }
\newcommand{\ip}[1]{\left\lfloor #1 \right\rfloor }

\newcommand{\fpnormal}[1]{\{ #1 \} }
\newcommand{\ipnormal}[1]{\lfloor #1 \rfloor }

\newcommand{\fpa}[1]{\left\lVert #1 \right\rVert_{\mathbb{R}/\mathbb{Z}}}
\newcommand{\floor}[1]{\left\lfloor #1 \right\rfloor}
\newcommand{\bra}[1]{\left(#1\right)}

\newcommand{\braBig}[1]{\Big( #1 \Big)}
\renewcommand{\tilde}{\widetilde}
\renewcommand{\bar}{\overline}
\newcommand{\nint}[1]{\left\lfloor #1 \right\rceil}

\newcommand{\abs}[1]{\left|#1\right|}

\newcommand{\set}[2]{\left\{ #1 \ \middle| \ #2 \right\} }

\newcommand{\ceil}[1]{\left\lceil #1 \right\rceil}

\renewcommand{\a}{\alpha}
\renewcommand{\b}{\beta}

\newcommand{\NN}{\mathbb{N}}

\newcommand{\QQ}{\mathbb{Q}}

\newcommand{\ZZ}{\mathbb{Z}}
\newcommand{\RR}{\mathbb{R}}
\newcommand{\CC}{\mathbb{C}}

\newcommand{\cC}{\mathcal{C}}

\newcommand{\cN}{\mathcal{N}}
\newcommand{\cO}{\mathcal{O}}

\definecolor{fresh}{HTML}{2bb101}
\definecolor{checked}{HTML}{1e5e06}
\definecolor{double}{HTML}{5E3800}
\definecolor{external}{HTML}{a81a78}
\definecolor{later}{HTML}{0410ff}
\definecolor{minor-rev}{HTML}{d96a09}
\definecolor{major-rev}{HTML}{c90000}
\definecolor{skip}{HTML}{ffffff}%
\definecolor{normal}{HTML}{000000}%

\definecolor{fresh}{HTML}{000000}
\definecolor{checked}{HTML}{000000}
\definecolor{double}{HTML}{000000}
\definecolor{external}{HTML}{000000}
\definecolor{later}{HTML}{000000}
\definecolor{minor-rev}{HTML}{000000}
\definecolor{major-rev}{HTML}{000000}

\renewcommand{\subset}{\subseteq}

\renewcommand{\leq}{\leqslant}
\renewcommand{\geq}{\geqslant}

\newcommand*\patchAmsMathEnvironmentForLineno[1]{%
  \expandafter\let\csname old#1\expandafter\endcsname\csname #1\endcsname
  \expandafter\let\csname oldend#1\expandafter\endcsname\csname end#1\endcsname
  \renewenvironment{#1}%
     {\linenomath\csname old#1\endcsname}%
     {\csname oldend#1\endcsname\endlinenomath}}%
\newcommand*\patchBothAmsMathEnvironmentsForLineno[1]{%
  \patchAmsMathEnvironmentForLineno{#1}%
  \patchAmsMathEnvironmentForLineno{#1*}}%
\AtBeginDocument{%
\patchBothAmsMathEnvironmentsForLineno{equation}%
\patchBothAmsMathEnvironmentsForLineno{align}%
\patchBothAmsMathEnvironmentsForLineno{flalign}%
\patchBothAmsMathEnvironmentsForLineno{alignat}%
\patchBothAmsMathEnvironmentsForLineno{gather}%
\patchBothAmsMathEnvironmentsForLineno{multline}%
}

\newcounter{claimcounter}
\counterwithin*{claimcounter}{theorem}

\newenvironment{claim*}{\refstepcounter{claimcounter}{\vspace{3pt}\par\noindent\textit{Claim:}}}{\vspace{3pt}}

\newenvironment{claimproof}[1]{\par\noindent\textit{Proof:}\space#1}{\hfill $\triangle$ \vspace{3pt}}

\newcommand{\Tr}{\operatorname{Tr}}
\newcommand{\Norm}{\operatorname{N}}

\begin{document}
\author[J.\ Byszewski ]{Jakub Byszewski}
\address[J.\ Byszewski]{Faculty of Mathematics and Computer Science,
Jagiellonian University\\
 \L{}ojasiewicza 6\\
30-348 Krak\'{o}w, Poland}
\email{jakub.byszewski@uj.edu.pl}

\author[J.\ Konieczny ]{Jakub Konieczny}
\address[J.\ Konieczny]{Universit\'e Claude Bernard Lyon 1, CNRS UMR 5208, Institut Camille Jordan,
F-69622 Villeurbanne Cedex, France}
\email{jakub.konieczny@gmail.com}

\title[{Pisot numbers, Salem numbers, and generalised polynomials}]{Pisot numbers, Salem numbers, \\ and generalised polynomials}

\date{\today}
\begin{abstract}
We study sets of integers that can be defined by the vanishing of a generalised polynomial expression. We show that this includes sets of values of  linear recurrent sequences of Salem type and some linear recurrent sequences of Pisot type. To this end, we introduce the notion of a generalised polynomial on a number field. We establish a connection between the existence of  generalised polynomial expressions for  sets of values of linear recurrent sequences and for subsemigroups of multiplicative groups of number fields.
\end{abstract}


\maketitle 

\section{Introduction}\label{sec:Intro}

Generalised polynomials are expressions built up from ordinary polynomials with the use of the integer part function, addition, and multiplication. In contrast with ordinary polynomials, generalised polynomial sequences can be bounded or even finitely-valued without being constant. For instance, for any irrational $\alpha \in(0,1)$ and any real $\beta$, the generalised polynomial map $g$ given by
\begin{equation}\label{eq:Int:sturm}
	g(n) = \floor{\a (n+1) + \b} - \floor{\a n + \b}
\end{equation}
defines a Sturmian sequence, which takes on only the values $0$ and $1$, with density $1-\alpha$ and $\alpha$, respectively.  
We define generalised polynomial sets to be the level sets of such maps. Equivalently, a generalised polynomial set $E\subset \ZZ$ is a set such that the characteristic function $1_E \colon \ZZ \to \{0,1\}$ is a generalised polynomial map.

It turns out that some sets of arithmetical or combinatorial interest are generalised polynomial sets. One example is the set of Fibonacci numbers, in which case an appropriate generalised polynomial can be constructed using the relation between the Fibonacci numbers and the golden mean together with some classical properties of continued fractions. It is a difficult problem to determine the extent to which this generalises to sequences $(n_i)_{i=0}^\infty$ that  satisfy a linear  recurrence 
\begin{equation}\label{eq:int:lin-rec}
	n_{i+m}=\sum_{j=0}^{m-1} a_{j} n_{i+j}, \quad i\geq 0,
\end{equation} 
for some $a_0,\ldots,a_{m-1} \in \ZZ$. One result in this direction concerns linear recurrent sequences whose characteristic polynomial is the minimal polynomial of a Pisot number. Recall that the characteristic polynomial of the recurrence \eqref{eq:int:lin-rec} is $X^m-\sum_{j=0}^{m-1} a_j X^j$, a (Galois) conjugate of an algebraic number $\beta$ is any root of the minimal polynomial of $\beta$ over $\QQ$, and a \emph{Pisot number} (or a Pisot--Vijayaraghavan number) is a real algebraic number $\beta$ such that $\beta > 1$, but all conjugates $\alpha$ of $\beta$ except for $\beta$ itself satisfy $\abs{\alpha} < 1$. An algebraic number is a \emph{unit} if both the number and its reciprocal are algebraic integers. An algebraic number is \emph{totally real} if all of its conjugates are real. The Dirichlet's unit theorem implies that for a real algebraic number $\beta$, the group of units $\cO_{\QQ(\b)}^*$ in $\QQ(\beta)$ has rank $1$ if and only if $\beta$ is either quadratic, or cubic and not totally real. The following theorem has been proved in many cases in \cite[Thm.\ B]{ByszewskiKonieczny-2018-TAMS} and  in full generality in \cite{AdamczewskiKonieczny}.

\begin{theorem}\label{thm:Pisot}
	Let $\beta$ be a Pisot unit such that $\cO_{\QQ(\b)}^*$ has rank $1$  and let $(n_i)_{i=0}^\infty$ be an integer-valued linear recurrent sequence with characteristic polynomial the minimal polynomial of $\beta$. Then the set $\set{n_i }{ i \in \NN_0}$ is generalised polynomial. 
\end{theorem}

It seems considerably more difficult to prove results in the opposite direction, that is, to establish that the set of values of a certain linear recurrent sequence is not generalised polynomial. Essentially the only known examples of such sequences have been obtained in \cite{Konieczny-2021-JLM}, where it is shown that the set $\{ k^i \mid i \in \NN_0\}$ is not generalised polynomial for any integer $k\geq 2$. Note that $k$ is a Pisot \emph{number}, but it is not a Pisot \emph{unit}.

In this paper, we obtain several extensions of Theorem \ref{thm:Pisot}. 
The first of them concerns Salem numbers. Recall that a real algebraic number $\beta$ is a \emph{Salem number} if $\beta > 1$, all conjugates $\alpha$ of $\beta$ except for $\beta$ itself satisfy $\abs{\alpha} \leq 1$, and there exists at least one  conjugate $\alpha$ with $\abs{\alpha} = 1$. If $\beta$ is a Salem number, then $1/\beta$ is a conjugate of $\beta$, and for all remaining  conjugates $\alpha$ we have $\abs{\alpha} = 1$ \cite[Lem.\ 1]{Smyth-2015}. For background on Salem numbers, we refer to \cite{Bertin-book} and \cite{Smyth-2015}. 

\begin{alphatheorem}\label{thm:Salem}
	Let $\beta$ be a Salem number and let $(n_i)_{i=0}^\infty$ be an integer-valued linear recurrent sequence with characteristic polynomial the minimal polynomial of $\beta$. Then the set $\set{n_i}{i \in \NN_0}$ is generalised polynomial. 
	\end{alphatheorem}

The proof of this result is most naturally phrased in terms of the notion of a \emph{generalised polynomial map on a number field}. A number field $K$ is a finite extension of $\QQ$, and generalised polynomials on $K$ can be defined in terms of the coordinates of an element in some $\QQ$-basis of $K$; for example, a generalised polynomial map $g$ on $\QQ(\sqrt{2})$ is of the form $g(x + y\sqrt{2}) = h(x,y)$, where $h$ is a generalised polynomial expression in two variables $x,y$ taking values in $\QQ$. We carefully introduce this concept in Section 2. 
Once the notion has been introduced, it is rather immediate to see  that the set $\{\beta^i \mid i\in \NN_0\}$ of powers of a Salem number $\beta$ is a generalised polynomial subset of $\QQ(\beta)$, and Theorem \ref{thm:Salem} can be deduced from this.

The special role of Pisot and Salem numbers in diophantine approximation is well recognized, even as many of the characterisations of Pisot and Salem numbers by their diophantine properties remain conjectural. In the context of generalised polynomials, we believe that Theorems \ref{thm:Pisot} and \ref{thm:Salem} should provide an essentially complete list of linear recurrent sequences whose set of values is generalised polynomial, and, similarly, an essentially complete list of algebraic numbers $\beta$ such that the set of powers $\{\beta^i \mid i\in \NN_0\}$ is a generalised polynomial subset of $\QQ(\beta)$. The following result elucidates the connection between these two questions.

\begin{alphatheorem}\label{thm:Tr-vs-Pow}
	Let $\beta$ be an algebraic integer. Suppose that there exists an integer-valued sequence $(n_i)_{i=0}^{\infty}$ with characteristic polynomial the minimal polynomial of $\beta$ that is not identically zero and is  such that the  set $\{n_i  \mid i \in \NN_0\}$ is generalised polynomial. Then the set
	$ \set{\beta^i}{i \in \NN_0}$  is a generalised polynomial subset of $\QQ(\beta)$. 
\end{alphatheorem}

The interest in the above result arises from the fact that it is likely  easier to show that the set of powers of an algebraic number is not generalised polynomial than to show that the corresponding result holds for the set of values of a linear recurrent sequence. In particular, the methods of \cite{Konieczny-2021-JLM} relied strongly on the fact that the set of powers of an integer $k$ forms a semigroup, and could conceivably be generalised.

Returning to Pisot numbers, we observe that we can strengthen the result obtained in Theorem \ref{thm:Pisot}. We say that a set $E$ is \emph{hereditarily generalised polynomial} or  that a generalised polynomial set is \emph{hereditary} if each subset $E' \subset E$ is generalised polynomial (this notion applies to generalised polynomial subsets of integers, number fields, etc.). The following result shows that the the sets considered in Theorem \ref{thm:Pisot} are in fact hereditary. In particular, any set consisting of Fibonacci numbers is generalised polynomial.

\begin{alphatheorem}\label{thm:Pisot-hereditary-Z}
	Let $\beta$ be a Pisot unit such that $\cO_{\QQ(\b)}^*$ has rank $1$, and let $(n_i)_{i=0}^\infty$ be an integer-valued linear recurrent sequence with characteristic polynomial the minimal polynomial of $\beta$. Let $I$ be an arbitrary subset of $\NN_0$. Then the set $\set{n_i}{i \in I}$ is generalised polynomial. 
\end{alphatheorem}

The above result has the following counterpart for sets of powers of $\beta$.

\begin{alphatheorem}\label{thm:Pisot-hereditary-K}
	Let $\beta$ be a Pisot unit such that $\cO_{\QQ(\b)}^*$ has rank $1$. Let $I$ be an arbitrary subset of $\NN_0$. Then the set $\set{\beta^i}{i \in I}$ is a generalised polynomial subset of $\QQ(\b)$. 
\end{alphatheorem}

In light of Theorems \ref{thm:Pisot-hereditary-Z} and \ref{thm:Pisot-hereditary-K}, the following question arises naturally.

\begin{question}
	Are the generalised polynomial sets considered in Theorem \ref{thm:Salem} hereditary?
\end{question}

It would be interesting to determine more generally which generalised polynomial sets are hereditary. Since a generalised polynomial subset of the integers always has density, no positive density generalised polynomial set of integers can be hereditary (see Section \ref{sec:No-Heir} for details), and so the question is only interesting  for sets with density zero. The task of disproving that a set is hereditarily generalised polynomial is made difficult by the fact that most tools available for showing that a given set $E \subset \ZZ$ with density zero is not generalised polynomial also yield the same conclusion for all supersets $E'$ of $E$ with density zero (see e.g.\ \cite[Thm.\ 3.1]{ByszewskiKonieczny-2018-TAMS}). Nevertheless, it is not true that every generalised polynomial set of density zero is hereditary.

\begin{alphatheorem}\label{thm:non-hereditary}
	There exists a set $E \subset \ZZ$ of density zero that is generalised polynomial but is not hereditary.
\end{alphatheorem}

In the context of this paper, it is natural to consider generalised polynomial expressions with variables taking rational, rather than integer, values. This corresponds to the notion of a generalised polynomial subset of $\QQ$ (rather than $\ZZ$). In the introduction we have for simplicity stated the main results  for subsets of $\ZZ$, rather than $\QQ$. This distinction is of no significance for Theorems \ref{thm:Salem} and \ref{thm:Pisot-hereditary-Z}, since any rational-valued linear recurrent sequence whose characteristic polynomial has integer coefficients is a rational multiple of an integer-valued sequence. This is not the case, however, for Theorem \ref{thm:Tr-vs-Pow}, and the formulation of this result in Theorem \ref{thm:Tr-vs-Pownew} below is genuinely more general, and applies also to sequences of rational numbers such as $(3^i/2^i)_{i=0}^{\infty}$; the corresponding $\beta=3/2$ is an algebraic number, but not an algebraic integer. 

The plan of the paper is as follows. In Section 2, we introduce the notion of a generalised polynomial map on a number field as well as its basic properties. In Sections 3 and 4, we study linear recurrent sequences arising from Pisot numbers, and we prove Theorems \ref{thm:Pisot-hereditary-K} (see Theorem \ref{thm:Pisot-hereditary-Knew}) and \ref{thm:Pisot-hereditary-Z} (see Theorem \ref{thm:Pisot-hereditary-Znew}). In Section 5, we obtain a similar result for Salem numbers (Theorem \ref{thm:Salem}, see Theorem \ref{thm:Salemnew}). In Section 6, we use trace maps and finiteness results for $S$-unit equations to prove Theorem \ref{thm:Tr-vs-Pow}. Finally, in Section 7, we construct an example of a generalised polynomial subset of $\ZZ$ that is not hereditary (Theorem \ref{thm:non-hereditary}, see Theorem \ref{thm:non-hereditarynew}).

\subsection*{Notation}

We let $\NN = \{1,2,3,\dots\}$ denote the set of positive integers and $\NN_0 = \NN \cup \{0\}$ the set of nonnegative integers. For a real number $x$, we let $\ip{x}$, $\ceil{x} = -\ip{-x}$, and $\nint{x} = \ip{x+1/2}$ denote the floor, the ceiling, and the nearest integer. We also let $\fp{x} = x - \ip{x}$ and $\fpa{x} = \min
\{\fp{x},1-\fp{x}\}$ denote the fractional part and the distance to the nearest integer. All of these expressions are generalised polynomials in $x$.

\subsection*{Acknowledgements} The authors wish to thank Boris Adamczewski for helpful comments. The first-named author was supported by National Science
Centre, Poland grant number 2018/29/B/ST1/01340. The second-named author works within the framework of the LABEX MILYON (ANR-10-LABX-0070) of Universit\'e de Lyon, within the program "Investissements d'Avenir" (ANR-11-IDEX-0007) operated by the French National Research Agency (ANR).

\section{Generalised polynomial maps on number fields}

In this section, we introduce the notion of a generalised polynomial map in a couple of related contexts, that is,  for maps defined on finite dimensional real vector spaces and for maps defined on number fields. The latter notion is new, and we carefully discuss its basic properties.

\subsection{Finite-dimensional real vector spaces}

Let $V$ be a finite dimensional real vector space. The class of  real-valued generalised polynomial maps $f\colon V \to \RR$ is the smallest class of functions containing constant maps and linear functionals, and closed under addition, multiplication, and taking the integer part of a map, that is, replacing $f$ by the map $\lfloor f \rfloor$ given by $\lfloor f \rfloor (x) =\lfloor f(x)\rfloor$.

A complex-valued map $f\colon V\to \CC$ is a generalised polynomial map if the real and imaginary parts of $f$ are real-valued generalised polynomial maps. We can give an equivalent characterisation of this class as follows. Let $\lfloor \cdot\rfloor_{\CC}\colon \CC \to \ZZ[i]$ be the complex integer part (or complex floor), defined by the formula $\lfloor z\rfloor_{\CC} = \lfloor \Re z\rfloor + i \lfloor \Im z \rfloor$. One then easily checks that the class of generalised polynomial maps $f\colon V \to \CC$ is the smallest class of functions containing complex-valued constant maps, (real-valued) linear functionals, and closed under addition, multiplication, and taking the complex integer part of a map, that is, replacing $f$ by the map $\lfloor f \rfloor_\CC$ given by $\lfloor f \rfloor_{\CC} (x) = \lfloor f(x) \rfloor_{\CC}$.

\subsection{Number fields}

In this subsection we introduce the notion of a generalised polynomial map defined on a number field. Even if a number field $K$ is given as a subfield of the complex (or real) numbers, these maps are $\emph{not}$ defined as restrictions of generalised polynomial maps on $\CC$; instead, this class consists, roughly speaking, of a much wider family of maps that can be expressed using the basic algebraic operations (addition and multiplication), the (complex) floor function, complex constants, and arbitrary embeddings of $K$ into the complex numbers. Since the floor function and the fractional part function can easily be expressed in terms of each other, replacing the floor  with the fractional part leads to an alternative definition of the same class. An example of a generalised polynomial map on $K=\QQ(\sqrt{2})$ is given by $a+b\sqrt{2} \mapsto \{a-b\sqrt{2}\}$. This map cannot be obtained as a restriction to $\QQ(\sqrt{2})$ of a generalised polynomial map on $\RR$ since it has infinitely many discontinuities in the interval $(0,1)$, which is not possible for a generalised polynomial map on $\RR$.

We now state the definition. Let $K$ be a number field. Consider the real vector space $K_{\RR} = K \otimes_{\QQ} \RR$ with the embedding $\iota \colon K \to K_{\RR}, \iota(x)=x\otimes 1$.  Describing this embedding $\iota$ concretely, we get the usual map $$K \to \RR^{r_1} \times \CC^{r_2},\quad  x\mapsto (\sigma_1(x),\ldots,\sigma_{r_1}(x),\tau_{1}(x),\ldots,\tau_{r_2}(x)),$$ where $\sigma_1,\ldots,\sigma_{r_1}$ are all the real embeddings of $K$, and $\tau_1,\bar{\tau_1},\ldots, \tau_{r_2},\bar{\tau}_{r_2}$ are all the complex embeddings of $K$, grouped in pairs. A map $f\colon K \to \CC$ is a generalised polynomial map if there exists a generalised polynomial map $\tilde{f}\colon K_{\RR} \to \CC$ (defined on the finite dimensional real vector space $K_{\RR}$) such that $f=\tilde{f} \circ \iota$. For a number field $L$ (regarded as a subfield of $\CC$) by a generalised polynomial map $f \colon K \to L$ we simply mean a generalised polynomial map $f \colon K \to \CC$ whose image $f(K)$ is contained in $L$.

In the following proposition we list some basic properties of generalised polynomial maps on number fields.

\begin{proposition}\label{prop:gpnf}
Let $K$ be a number field.
\begin{enumerate}
\item \label{prop:gpnf1} The class of generalised polynomial maps $f\colon K \to \CC$ is the smallest class that contains constant maps, field embeddings $\sigma \colon K \to \CC$, and is closed under addition, multiplication, and taking the complex integer part.

\item \label{prop:gpnf6} A map $f\colon K \to \CC$ is a generalised polynomial on $K$ if and only if its real and imaginary parts are generalised polynomial maps on $K$. 

\item \label{prop:gpnf2} If $f\colon K \to \CC$ is a generalised polynomial map on $K$ and $\varphi\colon \CC \to \CC$ is any field automorphism of $\CC$, then $\varphi \circ f \colon K \to \CC$ is also a generalised polynomial map on $K$. 
\item \label{prop:gpnf3} If $f\colon K \to \CC$ is a generalised polynomial map on $K$ and $g\colon L \to K$ is a generalised polynomial map on a number field $L$ taking values in $K$, then $f\circ g \colon L \to \CC$ is a generalised polynomial map on $L$.
\item \label{prop:gpnf4} If $f\colon K \to L$ is a generalised polynomial map on $K$ taking values in a number field $L$ and $\alpha_1,\ldots,\alpha_m$ is a basis of $L$ over $\QQ$, then there exist generalised polynomial maps $f_i\colon K \to \QQ$ on $K$ such that $f=\sum_i \alpha_i f_i$.  
\item \label{prop:gpnf5} If $f\colon K\to \CC$ is a generalised polynomial map on $K$, then the map $g\colon K\to \CC$ given by $$g(x)=\begin{cases} 1&\text{if } f(x)=0;\\ 0&\text{otherwise}\end{cases}$$ is a generalised polynomial on $K$. 

\end{enumerate}

\end{proposition}

\begin{proof} The claims in \ref{prop:gpnf1} and \ref{prop:gpnf6} follow from a similar claim for generalised polynomial maps on $K_{\RR}$.

To prove \ref{prop:gpnf2}, fix an automorphism $\varphi$ of $\CC$, and consider the class of maps $f\colon K \to \CC$ such that $\varphi\circ f$ is a generalised polynomial map on $K$. It is clear that this class contains constant maps, field embeddings, and is closed under addition, multiplication, and the complex integer part $\lfloor \cdot \rfloor_{\CC}$ (the latter property is due to the fact that $\sigma \circ \lfloor f \rfloor_{\CC}$ is either $\lfloor f\rfloor_{\CC}$ or $\overline{ \lfloor f\rfloor_{\CC}}$, depending on whether $\sigma(i)=i$ or $\sigma(i)=-i$). Thus, the claim follows from \ref{prop:gpnf1}.

To prove \ref{prop:gpnf3}, fix a generalised polynomial map $g\colon L \to K$ and consider the family of maps $f\colon K\to \CC$ such that $f\circ g$ is a generalised polynomial map on $L$. Using \ref{prop:gpnf1}, \ref{prop:gpnf2} and the fact that each complex embedding of $K$ can be extended to an automorphism of $\CC$, we verify that this family contains all generalised polynomial maps on $K$.

To prove \ref{prop:gpnf4}, we regard $L$ as a subfield of $\CC$. Let $\sigma_1,\ldots,\sigma_m$ denote all the embeddings of $L$ into $\CC$, and extend them in an arbitary way to automorphisms of $\CC$ (denoted by the same letter). We can uniquely write  $f$ in the form $f=\sum_i \alpha_i f_i$ for some maps $f_i:K\to \QQ$. We need to prove that $f_i$ are generalised polynomial maps on $K$. Applying the automorphism $\sigma_j$ to the above equality, we get $$\sigma_j \circ f = \sum_i \sigma_j(\alpha_i) f_i, \qquad 1\leq j \leq m.$$  The matrix $[\sigma_j(\alpha_i)]_{1\leq i,j\leq m}$ is nonsingular (see e.g.\ \cite[VI, \S 4]{Langbook}), and inverting the matrix, we can write $f_i$ as linear combinations of $\sigma_j \circ f$. Thus, the fact that $f_i$ are generalised polynomial maps on $K$ follows from \ref{prop:gpnf1} and \ref{prop:gpnf2}.

To prove \ref{prop:gpnf5}, note first that \ref{prop:gpnf6} reduces the claim to the case where $f$ takes real values (with the map $g$ equal to the product of the maps corresponding to the real and imaginary parts of $f$). A real number $y$ is zero if and only if both $y$ and $\sqrt{2} y$ are integers; thus,
\[
	g(x) = \ip{ 1-\fp{f(x)} } \cdot \ipnormal{ 1- \fpnormal{\sqrt{2} f(x)}},
\]
and so $g$ is a generalised polynomial map.\end{proof}

\subsection{Sets of algebraic numbers}

We say that a subset $S$ of a number field $K$ is a generalised polynomial subset of $K$ if its characteristic function $1_S \colon K \to \RR$ is a generalised polynomial map on $K$. We should pose a warning here: sometimes, one talks about generalised polynomial subsets of $\RR$ or $\CC$; these are defined as the zero sets of generalised polynomial maps defined on (real vector spaces) $\RR$ or $\CC$. However, even when the number field $K$ is given as a subfield of $\RR$ or $\CC$, these notions do not coincide! In fact, in the above sense no number field $K$ is a generalised polynomial subset of $\CC$, while a number field $K$ is clearly a generalised polynomial subset of itself. Of course, a generalised polynomial subset of $\RR$ that happens to be contained in a number field $K$ (for example, $\QQ$) is a generalised polynomial subset of $K$.  For this reason, in this paper we shall not talk about generalised polynomial subsets of $\RR$ or $\CC$, but only about generalised polynomial subsets of number fields (or, later, algebraic numbers).

 Since $1_{S\cap T}=1_S 1_T$ and $1_{K\setminus S}=1_K-1_S$, the class of generalised polynomial subsets of $K$ is closed under finite unions, finite intersections, and complements. Proposition \ref{prop:gpnf}\ref{prop:gpnf5} says that the zero set of a generalised polynomial map $f\colon K\to \CC$ is a generalised polynomial set. Moreover, whether a set is generalised polynomial or not is invariant under translation, applying a bijective $\QQ$-linear map (or, more generally, a generalised polynomial bijection $K\to K$ with generalised polynomial inverse), as well as adding or removing finitely many elements. In the following proposition we list some examples of generalised polynomial subsets of number fields.
 
\begin{proposition}\label{prop:gpsetsexam}
Let $K$ be a number field. The following subsets of $K$ are generalised  polynomial: \begin{enumerate}
\item\label{prop:gpsetsexam0} any $\QQ$-subvector space $V$ of $K$;
\item\label{prop:gpsetsexam1}
 any subfield $L\subset K$;
\item\label{prop:gpsetsexam2}  any lattice $\Lambda\subset K$;
\item\label{prop:gpsetsexam2a} the ring $\mathcal{O}_K$ of algebraic integers in $K$; 
\item\label{prop:gpsetsexam3}  
the group of units $\mathcal{O}_K^*$;
\item\label{prop:gpsetsexam3a}
any finite index subgroup of $\mathcal{O}_K^*$; 
\item\label{prop:gpsetsexam4} 
the set of Pisot units in $K$;
\item\label{prop:gpsetsexam6} 
the set of Salem numbers in $K$.
\end{enumerate}
 (The statements in \ref{prop:gpsetsexam4} and \ref{prop:gpsetsexam6} only make sense when $K$ is given a subfield of $\RR$.)
\end{proposition}
\begin{proof}

For \ref{prop:gpsetsexam0}, we first note that since the family of generalised polynomial subsets is closed under taking finite intersections, we may assume that $V$ is of codimension $1$; moreover, since the notion is stable under applying a bijective $\QQ$-linear map, it is sufficient to prove that a single codimension $1$ subspace $V$ is generalised polynomial. We may thus choose $V$ to be $\{x\in K \mid \Tr_{K/{\QQ}}(x)=0\}$, in which case the claim follows from Proposition \ref{prop:gpnf}\ref{prop:gpnf5}, since $\Tr_{K/{\QQ}}$ is a generalised polynomial map on $K$. This proves \ref{prop:gpsetsexam0}, and  \ref{prop:gpsetsexam1} is an immediate corollary.

For \ref{prop:gpsetsexam2}, choose a basis $v_1,\ldots,v_m$ of $\Lambda$, and extend it to a basis $v_1,\ldots,v_n$ of $K_{\RR}$. Let $v_1^*,\ldots,v_n^*$ denote the dual basis, i.e.\ linear maps $v_i^* \colon K_{\RR} \to \RR$ such that $v_i^*(v_i) = 1$ and $v_i^*(v_j) = 0$ for $1 
\leq i,j \leq m$ with $i \neq j$. Then $\Lambda$ is the common set of zeros of the generalised polynomial maps $\fp{v_1^*},\fp{v_2^*}, \ldots \fp{v_m^*}, v_{m+1}^*, v_{m+2}^*, \ldots, v_n^*$.
Item \ref{prop:gpsetsexam2a} follows directly from \ref{prop:gpsetsexam2} since $\cO_K$ is a lattice. 

For \ref{prop:gpsetsexam3}, we characterise the units as algebraic integers $\alpha$ of norm $\Norm_{K/\QQ}(\alpha)=\pm 1$. For \ref{prop:gpsetsexam3a}, let $H$ be a subgroup of $\mathcal{O}_K^*$ of finite index. By Chevalley's theorem \cite[Thm.\ 1]{Chevalley51}, $H$ is a congruence subgroup, meaning that $H=\mathcal{O}_K^* \cap \Lambda$ for some $\Lambda$ that is a union of finitely many cosets of a lattice. It remains to recall that $\mathcal{O}_K^*$ and $\Lambda$ are generalised polynomial.

For \ref{prop:gpsetsexam4}, we characterize Pisot units $\alpha$ in $K$ by requiring that: i) $\alpha$ be a unit; ii) $\alpha$ be positive; and iii) $|\sigma({\alpha})|^2 = \sigma(\alpha)\bar{\sigma}(\alpha)$ lie in the interval $(0,1)$ for all nonidentity embeddings $\sigma \colon K\to \CC$. The fact that the first two conditions define a generalised polynomial subset follows from \ref{prop:gpsetsexam3}, since positive units form a subgroup of the group of units of index $2$; the last condition defines the common zero set of the generalised polynomial maps $n\mapsto \lfloor \sigma(n)\bar{\sigma}(n)\rfloor$ (with the element $0$ removed). 

For \ref{prop:gpsetsexam4}, we similarly characterise Salem numbers $\alpha$ in $K$ by requiring that i) $\alpha$ be a unit; ii) $\alpha$ be positive;  iii) $|\sigma(\alpha)|^2=1$ for all but two embeddings $\sigma$;  iv) $1/\alpha$ be an algebraic conjugate of $\alpha$.  For the first three conditions, we apply the same reasoning as before. The fourth condition says that there exist two embeddings $\sigma$ and $\tau$ of $K$ is $\CC$ with $\sigma(\alpha) = \tau(\alpha)^{-1}$, which is also easily expressed in terms of generalised polynomial maps. 
\end{proof}

It is worthwhile to note that the set of Pisot \emph{numbers} is in general not a generalised polynomial subset; in fact, when $K=\QQ$, the set of Pisot numbers is simply the set of integers $\geq 2$, which is not generalised polynomial; this can be inferred, for instance, from the general fact that for a generalised polynomial set $E \subset \ZZ$ the limit $\abs{E \cap [M,M+N)}/N$ converges uniformly in $M$ as $N \to \infty$; cf.\ \cite[Ex.\ B.2]{Konieczny-2021-JLM}. On the other hand, the set of Pisot numbers and their negatives is a generalised polynomial subset (by a similar argument as for Pisot units).

Let $\QQ^{\mathrm{alg}}$ denote the field of algebraic numbers. We say that a subset $S$ of $\QQ^{\mathrm{alg}}$ is generalised polynomial if $S\cap K$ is a generalised polynomial subset of $K$ for every number field $K$. From Proposition \ref{prop:gpsetsexam} we get that if $S$ is itself a subset of some number field $L$, then $S$ is a generalised polynomial subset of $\QQ^{\mathrm{alg}}$ if and only if it is a generalised polynomial subset of $L$. Proposition \ref{prop:gpsetsexam} also immediately implies the following result.

\begin{proposition}
The following sets of algebraic numbers are generalised polynomial:
\begin{enumerate} \item the ring of algebraic integers; \item the group of algebraic units; \item the set of Pisot units; \item the set of Salem numbers. \end{enumerate}
\end{proposition}

\section{Pisot numbers: Number fields}\label{sec:Pisot:K}

In this section we prove Theorem \ref{thm:Pisot-hereditary-K}. We begin with a basic observation.

\begin{lemma}\label{lem:subset:all-beta^i}
Let $\beta$ be a Pisot unit and let $K=\QQ(\beta)$. Assume that $\cO_{K}^*$ has rank $1$. Then the set $\set{\beta^i}{i \in \NN_0}$ is a generalised polynomial subset of $K$. 
\end{lemma}
\begin{proof}
The set $\set{\beta^i}{i \in \ZZ}$ is a finite-index subgroup of the group of units $\cO_{K}^*$, and hence it is a generalised polynomial subset of $K$ by Proposition \ref{prop:gpsetsexam}\ref{prop:gpsetsexam3a}. The set of all Pisot units in $K$ is also generalised polynomial by Proposition \ref{prop:gpsetsexam}\ref{prop:gpsetsexam4}. It remains to observe that the intersection of these two sets is $\set{\beta^{i}}{i \in \NN}$.
\end{proof}

For technical reasons, it is easier to prove Theorem \ref{thm:Pisot-hereditary-K} in the case where $\beta$ is sufficiently large. Thus, we will first prove an analogous result with $\beta$ replaced with a sufficiently large power $\beta^m$.

\begin{proposition}\label{prop:subset:gamma^i-I}
Let $\beta$ be a Pisot unit and let $K=\QQ(\beta)$. Assume that $\cO_{K}^*$ has rank $1$. Then there exists $m_0$ such that for every integer $m \geq m_0$ and set $I \subset \NN_0$, the set $\set{\beta^{im}}{i \in I}$ is a generalised polynomial subset of $K$.
\end{proposition}
\begin{proof}

Since $\beta$ is a Pisot number, we can find $\rho > 1$ such that $\rho < \beta$ and for all conjugates $\alpha$ of $\beta$ other than $\beta$ itself we have  $\abs{\alpha} < 1/\rho$. Let $m \geq m_0$ be a large integer, where $m_0>0$ remains to be determined in the course of the argument, and let $\gamma := \beta^m$.
Consider the real number
\begin{equation}\label{eq:sub:01:xi}\
	\xi = \sum_{i \in I} \frac{\beta}{\gamma^{i}}.
\end{equation}	
For any integer $i$, the trace $\Tr_{K/\QQ}(\beta^i)$ is an integer. This implies that $\fpa{\beta^i}=O(1/\rho^{i})$ for $i\geq 0$; on the other hand, for $i\leq 0$ we trivially have $\fpa{\beta^i}=O(1/\rho^{|i|})$. Thus, for $i,j \in \NN$, we have the estimates
\begin{equation}\label{eq:sub:01:01}
	\fpa{\beta \gamma^{j-i}} = \fpa{\beta^{1 + m(j-i)}}= 
	\begin{cases}
		\fpa{\beta} &\text{if } i = j;\\
		O(1/\rho^{m|i-j|}) &\text{if } i \neq j,
	\end{cases}
\end{equation}	
where the constant implicit in the $O(\cdot)$-notation depends on $\beta$, but not on $m$. As a consequence, taking the sum over all $i \in I$, we obtain
\begin{equation}\label{eq:sub:01:n-xi}
	\fpa{\gamma^j \xi} 
	= 1_{I}(j) \fpa{\beta} + O\braBig{\sum_{i \in I \setminus \{j\}} 1/\rho^{m\abs{i-j}}} 
	= 1_{I}(j) \fpa{\beta} + O\bra{1/\rho^{m}}.
\end{equation}	
Assume that $m$ is large enough so that the error term in \eqref{eq:sub:01:n-xi} is strictly smaller than $\fpa{\beta}/3$.
Then for every $j \in \NN$ we have the equivalence
\begin{equation}\label{eq:sub:01:n-xi-<=>}
	j \in I \quad \text{if and only if}\quad \fpa{\gamma^j \xi} \geq \frac{2}{3}\fpa{\beta}.
\end{equation}
Let $g \colon K \to \{0,1\}$ be given by
\[
	g(x) = 
	\begin{cases}
		1 & \text{if } \fpa{\gamma^j x} \in \left[\frac{2}{3}\fpa{\beta},\frac{1}{2}\right];\\
		0 & \text{otherwise}.
	\end{cases}
\] 
We deduce from Proposition \ref{prop:gpnf}\ref{prop:gpnf5} that $g$ is a generalised polynomial map. By Lemma \ref{lem:subset:all-beta^i}, the set $\set{\gamma^i}{i \in \NN_0}$ is a generalised polynomial subset of $K$. It follows from the preceding discussion that for all $x \in \set{\gamma^i}{i \in \NN_0}$ we have $g(x) = 1$ if and only if $x \in \set{\gamma^i}{i \in I} = \set{\beta^{im}}{i \in I}$. 
\end{proof}

\begin{theorem}[= Theorem \ref{thm:Pisot-hereditary-K}]\label{thm:Pisot-hereditary-Knew}
	Let $\beta$ be a Pisot unit such that $\cO_{\QQ(\b)}^*$ has rank $1$. Let $I$ be an arbitrary subset of $\NN_0$. Then the set $\set{\beta^i}{i \in I}$ is a generalised polynomial subset of $\QQ(\b)$. 
\end{theorem}

\begin{proof}
	Let $m$ be an integer that is sufficiently large for the conclusion of Proposition \ref{prop:subset:gamma^i-I} to hold. Since generalised polynomial sets are closed under finite unions, it suffices to show that for each $0 \leq a <m$, the set $\set{\beta^i}{i \in I,\ i \equiv a \bmod{m}}$ is generalised polynomial. Since generalised polynomial sets are also invariant under dilation, we may freely assume that $a = 0$. The statement now follows from Proposition \ref{prop:subset:gamma^i-I}.
\end{proof}

\begin{remark}
	Similar techniques could also be applied in the situation where $\beta$ is an arbitrary Pisot unit without assumming that the unit group $\cO_{\QQ(\beta)}^*$ has rank $1$. However, in this case we no longer know whether or not the set $\set{\beta^i}{i \in \NN_0} \subset \QQ(\beta)$ is generalised polynomial. As a consequence, we would only obtain a relative result; namely, for each set $I \subset \NN_0$, the set $\set{\beta^i}{i \in I}$ is a generalised polynomial subset of $\set{\beta^i}{i \in \NN_0}$. In other words, there exists a generalised polynomial set $S \subset \QQ(\beta)$ such that for $i \in \NN_0$ we have $\beta^i \in S$ if and only if $i \in I$. Since it is not clear how interesting this generalisation is, we do not go into the details at this point.
\end{remark}

\section{Pisot numbers: Integers}\label{sec:Pisot:Z}

In this section, we prove Theorem \ref{thm:Pisot-hereditary-Z}. We first recall a well known fact on linear recurrent sequences.

\begin{lemma}\label{lem:lrseqform}
Let $\beta$ be an algebraic number, let $K=\QQ(\beta)$, and let $m=[K:\QQ]$ be the degree of $\beta$. Let $(n_i)_{i=0}^{\infty}$ be a linear recurrent sequence with rational values and with characteristic polynomial the minimal polynomial of $\beta$. \begin{enumerate} \item\label{lem:lrseqform1}  There exists a unique $x\in K$ such that $$n_i = \Tr_{K/\QQ}(\beta^i x) \quad \text{for all } i\geq 0.$$
\item\label{lem:lrseqform2} Assume moreover that  $(n_i)_{i=0}^{\infty}$ is not identically zero. Suppose that  $(n'_i)_{i=0}^{\infty}$ is another linear recurrent sequence with characteristic polynomial the minimal polynomial of $\beta$ and taking values in some extension $L$ of $\QQ$. Then the sequence $(n'_i)_{i=0}^{\infty}$ can be written as a linear combination of the sequences $(n_{i+j})_{i=0}^{\infty}$, $0\leq j<m$, with coefficients in $L$.\end{enumerate}
\end{lemma} 
\begin{proof}
The vector space $V$ of all rational-valued linear recurrent sequences with characteristic polynomial the minimal polynomial of $\beta$ is clearly $m$-dimensional. Since all sequences of the form $(\Tr_{K/\QQ}(\beta^i x))_{i=0}^{\infty}$ lie in this space, \ref{lem:lrseqform1} follows from the nondegeneracy of the bilinear map $(x,y)\mapsto \Tr_{K/\QQ}(xy)$. To prove \ref{lem:lrseqform2}, write $(n'_i)_{i=0}^{\infty}$ as an $L$-linear combination of sequences in $V$, and apply \ref{lem:lrseqform1}.
\end{proof}

 An elementary but key fact about integer-valued sequences satisfying a Pisot linear recurrence is that each successive term can be computed by a simple generalised polynomial formula involving only the previous term. We record this in the following lemma.

\begin{lemma}\label{lem:subset:n_i->n_i+1}
	Let $(n_i)_{i=0}^\infty$ be an integer-valued sequence satisfying a linear recurrence whose characteristic polynomial is the minimal polynomial of a Pisot number $\beta$. Then for each $j\geq 0$ there exists some $i_0$ such that for all integers $i\geq i_0$ we have 
	\(
		n_{i+j} = \nint{\beta^j n_i}.
	\)
\end{lemma} 
\begin{proof}
Lemma \ref{lem:lrseqform}\ref{lem:lrseqform1} allows us to write the sequence $(n_i)_{i=0}^{\infty}$ in the form
\begin{align}\label{eq:sub:34:0}
	n_i = \sum_{\a} c_{\a} \a^i,
\end{align}
where the sum runs over all conjugates $\a$ of $\b$, and $c_\a$ are complex constants. Since $\b$ is Pisot, we have $\a^i \to 0$ as $i \to \infty$ for each $\a \neq \b$, and hence 
\begin{align}\label{eq:sub:34:1}
		n_i - c_\b \beta^i \to 0 \quad \text{as } i \to \infty.
\end{align}
It follows that
\begin{align}\label{eq:sub:34:2}
		n_{i+j} - \beta^j n_i \to 0 \quad \text{as } i \to \infty.
\end{align}
Thus, $n_{i+j} = \nint{\beta^j n_i}$ for sufficiently large $i$.
\end{proof}

Lemma \ref{lem:subset:n_i->n_i+1} allows us to pass between the terms of any two linear recurrent sequences satisfying the same Pisot linear recurrence by applying a generalised polynomial map.

\begin{proposition}\label{lem:subset:n_i<->n_i'}
	Let $\beta$ be a Pisot number and let $K=\QQ(\beta)$. Let $(n_i)_{i=0}^\infty$ and $(n_i')_{i=0}^\infty$ be two sequences taking values in $\QQ(\beta)$ and satisfying a linear recurrence whose characteristic polynomial is the minimal polynomial of $\beta$. Assume also that $(n_i)_{i=0}^\infty$ is not identically zero. Then there exists a generalised polynomial map $g \colon \QQ(\beta) \to  \QQ(\beta)$ such that $g(n_i) = n_i'$ for all but finitely many positive integers $i$. 
\end{proposition}	
\begin{proof}
Replacing $n_i$ with $\Tr_{K/\QQ}(\xi n_i)$ for suitably chosen $\xi \in K$, we may freely assume that $n_i$ are integers for all $i$. Let $m$ denote the degree of $\beta$.  By Lemma \ref{lem:lrseqform}\ref{lem:lrseqform2}, we may express $(n_i')_{i=0}^{\infty}$  as
	\[
		n_{i}' = \sum_{j=0}^{m-1} w_j n_{i+j}\quad \text{for all } i \in \NN_0,
	\]
	where $w_j \in K$, $0 \leq j < m$, are some coefficients. It follows from Lemma \ref{lem:subset:n_i->n_i+1} that
	\[
		n_{i}' = \sum_{j=0}^{m-1} w_j \nint{\beta^j n_i} \quad\text{for all sufficiently large } i.	
	\]
	Thus, we may take $g(n) = \sum_{j=0}^{m-1} w_j \nint{\beta^j n}$.  
\end{proof}

We can now prove the following result, which is a slightly stronger version of Theorem \ref{thm:Pisot-hereditary-Z}.

\begin{theorem}\label{thm:Pisot-hereditary-Znew}
	Let $\beta$ be a Pisot unit such that $\cO_{\QQ(\b)}^*$ has rank $1$, and let $(n_i)_{i=0}^\infty$ be a linear recurrent sequence of rational numbers with characteristic polynomial the minimal polynomial of $\beta$. Let $I$ be an arbitrary subset of $\NN_0$. Then the set $\set{n_i}{i \in I}$ is a generalised polynomial subset of $\QQ$. 
\end{theorem}
\begin{proof}
	We may assume that $(n_i)_{i=0}^{\infty}$ is not identically zero. It follows from Proposition \ref{lem:subset:n_i<->n_i'} that there exists a generalised polynomial map $g \colon \QQ \to \QQ(\beta)$ such that $g(n_i) = \beta^i$ for all sufficiently large $i$. The set $\set{\beta^i }{i \in I} \subset \QQ(\beta)$ is a generalised polynomial set by Theorem \ref{thm:Pisot-hereditary-Knew}, and the set $\set{n_i}{i \in \NN_0}$ is generalised polynomial by Theorem \ref{thm:Pisot}. The claim follows from the fact that the class of generalised polynomial sets is stable under taking preimages by generalised polynomial maps, finite intersections, and finite modifications.
\end{proof}

\section{Salem numbers}\label{sec:Salem}

In this section, we prove Theorem \ref{thm:Salem}.
As we have already pointed out, the notion of a generalised polynomial subset is preserved by applying a bijective generalised polynomial map with generalised polynomial inverse. The following lemma records a similar principle, but allowing for the inverse to be defined on a case-by-case basis.
\begin{lemma}\label{lem:Sal:select}
	Let $K$ and $L$ be number fields, let $S \subset K$ be a generalised polynomial set, and let $f \colon K \to L$ and $g_1,g_2,\dots,g_r \colon L \to K$ be generalised polynomial maps. Suppose that for each $x \in S$ there exists $1 \leq i \leq r$ such that $g_i(f(x)) = x$. Then $f(S)$ is a generalised polynomial subset of $L$.
\end{lemma}
\begin{proof}
	Let $1 \leq i \leq r$ and put 
	\[
		S_i := \set{x \in S}{g_i(f(x)) = x}.
	\]
	It follows from Proposition \ref{prop:gpnf}\ref{prop:gpnf5} that $S_i$ is a generalised polynomial subset of $K$.  We claim that \begin{equation} \label{eqn51} f(S_i) = \set{y \in L}{ g_i(y) \in S_i,\ f(g_i(y)) = y}.\end{equation} Indeed, if $y \in f(S_i)$, say $y = f(x)$ for some $x \in S_i$, then $g_i(y) = g_i(f(x)) = x \in S_i$ and $f(g_i(y)) = f(g_i(f(x))) = f(x) = y$. Conversely, if $y \in L$, $g_i(y) \in S_i$ and $y = f(g_i(y))$ then clearly $y \in f(S_i)$. From \eqref{eqn51} we deduce that $f(S_i)$ is a generalised polynomial subset of $L$. Since $f(S) = \bigcup_{i=1}^r f(S_i)$, we conclude that $f(S)$ is a generalised polynomial subset of $L$.
\end{proof}

For linear recurrent sequences of Pisot type, we proved that one can pass between the corresponding terms of the sequences using a generalised polynomial map (see Proposition \ref{lem:subset:n_i<->n_i'}). For linear recurrent sequences of Salem type an analogous result holds if we allow instead the use of a finite family of generalised polynomial maps.

\begin{proposition}\label{lem:Sal:recover}
	Let $\beta$ be a Salem number and let $L$ be the splitting field of the minimal polynomial of $\beta$. Let $(n_i)_{i=0}^\infty$ be a sequence of rational numbers satisfying a linear recurrence whose characteristic polynomial is the minimal polynomial of $\beta$. Assume that $(n_i)_{i=0}^\infty$ is not identically zero. Then there exists a finite family $\{g_c\}_{c\in \cC}$ of generalised polynomial maps $g_c \colon \QQ \to L$ such that for each $i \in \NN_0$ there exists $c \in \cC$ such that $g_c(n_i) = \b^i$.\end{proposition}
\begin{proof}
	Let $i \in \NN_0$. Using Lemma \ref{lem:lrseqform}\ref{lem:lrseqform1}, we can write $n_i$ in the form $n_i=\Tr_{\QQ(\beta)/\QQ}(\beta^i x)$ for some $x\in \QQ(\beta)$, which  allows us to express $n_i$ in the form
\[
	n_i = \sum_{\a} w_\a \a^i,
\]	
	where the sum runs over all conjugates $\alpha$ of $\beta$ and $w_\a \in L$ are constants. Since $w_\a$ are images of $x$ by embeddings of $\QQ(\beta)$ into $ \CC$, and since $x\neq 0$ (otherwise the sequence $(n_i)_{i=0}^{\infty}$ would be identically zero), we see that $w_{\alpha}$ are all nonzero.	Since all $\alpha \neq \beta$ have absolute value $1$ or $1/\beta \leq 1$, we have for each $j \in \NN_0$ the estimate
\begin{align*}
	\abs{ \floor{\beta^j n_i} - n_{i+j} } 
	&\leq 1 + \abs{ \b^j \sum_{\a } w_\a \a^i - \sum_{\a } w_\a \a^{i+j}} 
	\\ &\leq 1 +  \sum_{\a \neq \b} \abs{w_\a} \bra{\b^j + 1} = O(\b^j),
\end{align*}
with the implicit constant depending on $\b$ and the sequence $(n_i)_{i=0}^\infty$, but not on $j$. 
It follows that there exist constants $c^{(i)}_j \in \ZZ$ with $\abs{c^{(i)}_j} = O(\beta^j)$ such that 
\begin{equation}
	 \floor{\beta^j n_i} - c^{(i)}_j = n_{i+j} = \sum_{\a} w_\a \a^{i+j}.
\end{equation}
Let $m$ be the degree of $\beta$ and consider the following system of linear equations in $x_{\alpha}$ and $y_j$:
\begin{equation}\label{eq:Sal:01:vandermonde}
	 y_j = \sum_{\a} w_\a \a^{j} x_\a \quad \text{for } 0 \leq j < m. 
\end{equation}
Note that \eqref{eq:Sal:01:vandermonde} holds for $x_\a = \a^i$ and $y_j = \floor{\b^j n_i} - c^{(i)}_j$. 
The determinant of the matrix $(w_{\alpha}\alpha^j)_{\alpha,j}$ is nonzero as the product of $w_{\alpha}$ (which are nonzero) and a Vandermonde's determinant. Thus \eqref{eq:Sal:01:vandermonde} has a unique solution, say
\begin{equation}\label{eq:Sal:01:sol-vdm}
	 x_{\a} = \sum_{j=0}^{m-1} \gamma_{j,\a} y_j,
\end{equation}
where $\gamma_{j,\a} \in L$ are some constants. Put $C_j := \max_{i} \abs{c_j^{(i)}} = O(\beta^j)$. Consider the set
\begin{equation}
\cC = \set{ (c_j)_{j=0}^{2d-1} }{ c_j \in \ZZ,\ \abs{c_j} \leq C_j \text{ for all } 0 \leq j < m}
\end{equation}
and for each $c = (c_j)_{j=0}^{2d-1} \in \cC$ the generalised polynomial
\begin{equation}\label{eq:Sal:01:def-g}
	g_c(n) = \sum_{j=0}^{m-1} \gamma_{j,\a}\bra{ \floor{\beta^j n} - c_{j} }.
\end{equation}
It follows from the preceding discussion that $\b^i = g_c(n_i)$ for $c \in \cC$ given by $c_j := c^{(i)}_j$, $0 \leq j < m$. \end{proof}

\begin{lemma}\label{lem:Salem:beta^j-is-GP}
	Let $K$ be a number field and let $\beta \in K$ be a Salem number. Then $\set{\beta^j}{j \in \NN_0} \subset K$ is a generalised polynomial subset of $K$.
\end{lemma}
\begin{proof}
	The set of Salem numbers in $\QQ(\beta)$ is of the form $\{\gamma^i \mid i\in \NN\}$ for some Salem number $\gamma$ \cite[p.\ 169]{Salem45}. From Chevalley's theorem we deduce that the group $\langle \beta \rangle$ is a congruence subgroup of $\langle \gamma \rangle$, and hence there exists some $\Lambda$ that is a union of finitely many cosets of a lattice with the property that $\langle \beta \rangle = \langle \gamma \rangle\cap \Lambda$. The fact that $\set{\beta^j}{j \in \NN_0}$ is generalised polynomial follows then from Proposition \ref{prop:gpsetsexam}.
\end{proof}

We can now deduce the following result, which is a slightly stronger version of Theorem \ref{thm:Salem}.

\begin{theorem}\label{thm:Salemnew}
	Let $\beta$ be a Salem number and let $(n_i)_{i=0}^\infty$ be a linear recurrent sequence  of rational numbers with characteristic polynomial the minimal polynomial of $\beta$. Then the set $\set{n_i}{i \in \NN_0}$ is a generalised polynomial subset of $\QQ$. 
	\end{theorem}
\begin{proof}
	Let $K=\QQ(\beta)$ and let $L$ be the splitting field of the minimal polynomial of $\beta$. We may suppose that the sequence $(n_i)_{i=0}^\infty$ is not identically zero. Using Lemma \ref{lem:lrseqform}\ref{lem:lrseqform1}, we write the sequence $(n_i)_{i=0}^\infty$ in the form $n_i = \Tr_{K/\QQ}(\beta^i x)$ for some $x\in K$. Let $S=\set{\beta^j}{j \in \NN_0} \subset K$, let $f\colon L \to \QQ$ be the map given on $K$ by $f(y)=\Tr_{K/\QQ}(yx)$, and extended arbitrarily to a generalised polynomial map on $L$, and let  $g_c \colon \QQ \to L$ be the maps satisfying the claim of Proposition \ref{lem:Sal:recover}. The result follows by applying  Lemma \ref{lem:Sal:select}  to $S$, $f$, and $\{g_c\}_{c\in \cC}$.
\end{proof}
\section{Generalised polynomial sets of powers}\label{sec:GPpowers}

In this section, we prove the following result, which is a stronger variant of Theorem \ref{thm:Tr-vs-Pow}. 

\begin{theorem}\label{thm:Tr-vs-Pownew}
	Let $\beta$ be an algebraic number. Suppose that there exists a linear recurrent sequence $(n_i)_{i=0}^{\infty}$ of rational numbers with characteristic polynomial the minimal polynomial of $\beta$ that is not identically zero and is  such that the  set of values $\{n_i  \mid i \in \NN_0\}$ is a generalised polynomial subset of $\QQ$. Then the set
	$ \set{\beta^i}{i \in \NN_0}$  is a generalised polynomial subset of $\QQ(\beta)$. 
\end{theorem}

Since the proof is somewhat lengthy and technical, we first sketch the main idea. Let $K=\QQ(\beta)$. The sequence $(n_i)_{i=0}^{\infty}$ can be written in the form $$n_i=\Tr_{K/\QQ}(\beta^i z)$$ for some $z\in K$. Let $X$ be the set of values of $(n_i)_{i=0}^{\infty}$; by assumption, $X$ is a generalised polynomial subset of $\QQ$.
We consider the set $Y$ of all elements $x\in K$ such that $\Tr_{K/\QQ}(\beta^k x z)$ belongs to $X$ for a large finite number of values of $k$, $0\leq k<N$.
Such a set is clearly a generalised polynomial subset of $K$. We might expect that any $x\in Y$ is necessarily of the form $x=\beta^i$ for some $i\geq 0$; if true, this would conclude the proof. Unfortunately, while this claim is quite close to being true, some caveats apply. First, some nondegeneracy conditions are required for $\beta$; the claim is usually false  if $\beta$ is a root of unity, e.g.\ for $\beta=i$, $z=1$, in which case  $$X=\{-2,0,2\} \qquad \text{and} \qquad Y=\{\pm 1, \pm i, \pm (1 +i), \pm (1-i)\} $$ provided that $N\geq 2$.
Perhaps less obviously, the claim is also false when $\beta$ has a conjugate of the form $\omega \beta$ with $\omega$ a root of unity, e.g. for $\beta=\sqrt{2}$, $z=1$, in which case $$X=\{0\}\cup\{2^{i+1}\mid i\in \NN_0\} \qquad \text{and} \qquad Y=\{0\}\cup\{2^{i}\mid i\in \NN_0\}\cup\{2^{i-1}\sqrt{2}\mid i\in \NN_0\} \cup \{2^{i} + 2^{j-1}\sqrt{2}\mid i,j\in \NN_0\} $$ provided that $N\geq 2$.
Second, the claim is false for a different reason if $\beta$ and $\beta^{-1}$ are conjugate, e.g. for $\beta=2+\sqrt{3}$, $z=1$, in which case $\Tr_{K/\QQ}(\beta^i)=\Tr_{K/\QQ}(\beta^{-i})$ and one can show that $$Y=\{\beta^i \mid i\in \ZZ\}$$ provided that $N\geq 3$. Finally, a finite number of exceptions are possible, namely $x=0$ and $x=\beta^i$ for some negative values of $i$. Nevertheless, with these three situtations properly accounted for, the claim becomes correct. To prove these results, we need to study when the values of traces coincide along certain geometric progressions.

We begin by recalling two well known facts, whose proofs we include for lack of appropriate reference.
\begin{lemma}\label{lem:conjugate-powers}
	Let $\beta$ and $\gamma$ be conjugate  algebraic numbers that are neither zero nor roots of unity. Let $k$ and $l$ be nonzero integers such that $\beta^k = \gamma^l$. Then $k=\pm l$ and $\gamma = \omega \beta^{\pm 1}$ for some root of unity $\omega$. 
\end{lemma}
\begin{proof}
Let $\sigma$ be an automorphism of the Galois closure of $\QQ(\beta)$ that maps $\beta$ to $\gamma$. For each integer $t$ we have $\sigma^t(\beta^l)=\sigma^{t-1}(\beta^k)$. Hence, setting $n$ to be the order of $\sigma$ and applying $n$ times the automorphism $\sigma$ to $\beta^{l^n}$, we get $$\beta^{l^n}=\sigma^n(\beta^{l^n})=\beta^{k^n}.$$ 
Since $\beta$ is neither zero nor a root of unity, it follows that $l=\pm k$. Thus, $(\beta \gamma^{\mp 1})^k = 1$, and $\beta \gamma^{\mp 1}$ is a root of unity.
\end{proof}

\begin{lemma}\label{lem:pow:equal-traces}
	Let $\beta$ be a nonzero algebraic number, let $m$ denote the degree of $\beta$, let $K = \QQ(\beta)$, and let $\gamma,x,y \in K$.
\begin{enumerate}
\item\label{it:pow:53:A} Suppose that \( \Tr_{K/\QQ}(\beta^{i}x) = \Tr_{K/\QQ}(\beta^{i}y) \)
	for $0 \leq i < m$. Then $x = y$.
\item\label{it:pow:53:B} Suppose that \(
		\Tr_{K/\QQ}(\beta^{i}x) = \Tr_{K/\QQ}(\gamma^{i}y)
\)	for $0 \leq i < 2m$ and that $x \neq 0$. Then there exists an automorphism $\sigma$ of $K$ such that $\sigma(\beta) = \gamma$ and $\sigma(x) = y$.
\end{enumerate}	
\end{lemma}
\begin{proof}

Item  \ref{it:pow:53:A} follows from the fact that $\Tr_{K/\QQ}$ induces a nondegenerate quadratic form on $K$. Hence, it remains to prove \ref{it:pow:53:B}.

Let $r, s \in \QQ[\![T]\!]$ be the generating functions associated to the sequences $(\Tr_{K/\QQ}(\beta^{i}x))_{i=0}^\infty$ and $(\Tr_{K/\QQ}(\gamma^{i}y))_{i=0}^\infty$, that is, $$r=\sum_{i\geq 0} \Tr_{K/\QQ}(\beta^{i}x) T^i, \qquad  s=\sum_{i\geq 0} \Tr_{K/\QQ}(\gamma^{i}y) T^i.$$ Since these sequences are linear recurrent and satisfy the same linear recurrences as $(\beta^i)_{i=0}^{\infty}$ and $(\gamma^i)_{i=0}^{\infty}$, respectively, we may write $r=p/f$, $s=q/g$, where $f$ is the minimal polynomial of $\beta$, $g$ is the minimal polynomial of $\gamma$, and $p,q\in \QQ[X]$ are polynomials of degree $\deg p <\deg f=m$, $\deg q <\deg g\leq m$. Our assumption guarantees that $pg-qf = fg(r-s)$, regarded as a power series in $T$, is divisible by $T^{2m}$. Since $pg-qf$ is also a polynomial of degree $< 2m$, we conclude that it is the zero polynomial, and hence $r=s$. Since $r$ is not identically zero, it follows that $f=g$, and so $\beta$ and $\gamma$ are conjugate. Let $\sigma$ the automorphism of $K$ such that $\sigma(\beta) = \gamma$. Then 
\[ 
\Tr_{K/\QQ}(\gamma^{i}\sigma(x)) = \Tr_{K/\QQ}(\beta^{i}x) = \Tr_{K/\QQ}(\gamma^{i}y).
\]
The equality $\sigma(x) = y$ now follows from \ref{it:pow:53:A}.
\end{proof}

We will use two fundamental (and related) results: the finiteness of the number of solutions of the $S$-unit equation and the Skolem--Mahler--Lech theorem. The first of these results was proved by Evertse \cite{Evertse84} and van der Poorten--Schlickewei \cite{vdPSch91}. In the formulation of \cite[Thm.\ 2]{vdPSch91} it says that if $K$ is a field of characteristic zero, $G$ is a finitely generated subgroup of the multiplicative group of $K$, and $a_1,\ldots,a_m$ are nonzero elements of $K$, then the equation $$\sum_{i=1}^n a_i g_i = 0$$ has, up to scaling, only finitely many solutions $(g_i)_{i=1}^n$ with $g_i \in G$ such that no proper sub-sum $\sum_{i \in I} a_i g_i$, $\emptyset \neq I \subsetneq\{1,\ldots,m\}$, vanishes; here, considering solutions up to scaling means that we identify solutions $(g_i)$ and $(g'_i)$ such that $g_i/g'_i$ is indepedent of $i$. 
The Skolem--Mahler--Lech theorem \cite{Skolem34, Lech53, Mahler56}  says that the set of zeros of a linear recurrent sequence over a field of characteristic zero is a union of a finite set and finitely many arithmetic progressions. This implies that a non-constant linear recurrent sequence whose characteristic polynomial is the minimal polynomial of an algebraic number $\beta$ has only finitely many zeros provided that $\beta$ satisfies the following nondegeneracy property:
\begin{equation} \label{eq:pow:non-degen}\tag{$\dagger$}
	\alpha/\alpha' \text{ is not  a root of unity for all conjugates } \a \neq \a' \text{ of } \beta. 
\end{equation}
(Alternatively, this could also be deduced directly from the $S$-unit equation.) Moreover, the number of zeros is bounded by a constant that depends only on the field $\QQ(\beta)$ \cite[Thm.\ 1.1]{Schlickewei-1996}.

The property \eqref{eq:pow:non-degen} will appear several times in the remainder of this section. We note that \eqref{eq:pow:non-degen} is equivalent to saying that $\QQ(\b) = \QQ(\b^d)$ for all $d \in \NN$. (This is because the set of conjugates of $\beta^d$ is equal to the set of $d$-th powers of conjugates of $\beta$.) Moreover, for each algebraic number $\gamma$, there exists an integer $d \in \NN$ such that \eqref{eq:pow:non-degen} holds for $\beta = \gamma^d$.

\begin{proposition}\label{prop:lemmasumsoftraces}
Let $\beta$ be an algebraic number satisfying \eqref{eq:pow:non-degen}, let $K = \QQ(\b)$, and let $a_0,\ldots,a_m \in K$. Then for all but finitely many nonnegative integer solutions $(n_0,n_1,\dots,n_m) \in \NN_0^{m+1}$ of 
\begin{equation}\label{eq:83:01}
\Tr_{K/\QQ}\bra{ \sum_{i=0}^m a_i \beta^{n_i}}=0
\end{equation}
there exists a nonempty subset $I$ of $\{0,1,\ldots,m\}$ such that 
\begin{equation}\label{eq:83:01a}
\sum_{i\in I} a_i\beta^{n_i}=0.
\end{equation}
Additionally, if $\beta$ has no conjugates of the form $\omega\beta^{-1}$ with $\omega$  a root of unity, then the same conclusion holds for solutions $(n_0,n_1,\dots,n_m) \in \ZZ^{m+1}$.
\end{proposition}

\begin{remark} Before proceeding with the proof, observe that the assumptions in Proposition \ref{prop:lemmasumsoftraces} are in fact necessary for the claims to hold.
In fact, if $\beta$ fails to satisfy condition \eqref{eq:pow:non-degen}, then there exists an integer $d\in \NN$ such that $\QQ(\beta^d)$ is a proper subfield of $K$, and we can find some nonzero $\gamma$ in the kernel of the map $\Tr_{K/\QQ(\beta^d)}\colon K\to \QQ(\beta^d)$. Let $m=[K:\QQ]-1$, and write $\gamma$ in the form $\gamma=\sum_{i=0}^{m}a_i\beta^i$, $a_i\in \QQ$. For any $n\in \NN_0$ we have  $$\Tr_{K/\QQ}(\sum_{i=0}^m a_i \beta^{i+nd})=\Tr_{K/\QQ}(\beta^{nd}\gamma) = \Tr_{\QQ(\beta^d)/\QQ}(\beta^{nd}\Tr_{K/\QQ(\beta^d)}(\gamma))=0.$$ 
Note that no nonempty sub-sum of  $\sum_{i=0}^m a_i\beta^{i+nd}$ vanishes. This contradicts the first claim of Proposition  \ref{prop:lemmasumsoftraces}.

Now suppose that $\beta$ has a conjugate of the form $\omega\beta^{-1}$ for some root of unity $\omega$. Then $\beta^{d}$ is conjugate to $\beta^{-d}$ for some $d\in \NN$, and so $$\Tr_{K/\QQ}(\beta^{nd}-\beta^{-nd})=0$$ for all $n\in \NN$. This contradicts the second claim of Proposition  \ref{prop:lemmasumsoftraces}. \end{remark}
\begin{proof}[Proof of Proposition \ref{prop:lemmasumsoftraces}]
Let $\cN \subset \ZZ^{m+1}$ be an arbitrary infinite family of solutions $(n_i)_{i=0}^m$ to \eqref{eq:83:01}. We further assume that either $\cN$ is a subset of $\NN_0^{m+1}$ or that $\beta$ has no conjugates of the form $\omega\beta^{-1}$ with $\omega$ a root of unity. Our aim is to show that under either of these assumptions we can find $(n_i)_{i=0}^m \in \cN$ and $\emptyset \neq I \subset \{0,1,\dots,m\}$ such that \eqref{eq:83:01a} holds.

 Let $\Sigma$ be the set of all embeddings of $K$ into $\CC$.
We may rewrite \eqref{eq:83:01} in the form 
\begin{equation}\label{eq:83:02}
\sum_{(i,\tau) \in \{0,\ldots,m\} \times \Sigma} \tau(a_i) \tau(\beta)^{n_i}=0.
\end{equation}
For each $(n_i)_{i=0}^m \in \cN$, we partition $\{0,\ldots,m\} \times \Sigma$  into pairwise disjoint nonempty sets $J$ that are minimal with respect to the property that
\begin{equation}\label{eq:83:03}
\sum_{(i,\tau) \in J}\tau(a_i) \tau(\beta)^{n_i} =0.
\end{equation}
Applying the pigeon-hole principle and replacing $\cN$ with an infinite subset, we may assume that this partition is the same for all $(n_i)_{i=0}^m \in \cN$. Let $J$ be a cell in this partition and let $I$ be the set of all $i \in \{0,\ldots,m\}$ such that $(i,\tau) \in J$ for at least one $\tau \in \Sigma$. Note that \eqref{eq:83:03} only depends on $n_i$ with $i \in I$. Choosing $J$ in a judicious manner, we can also ensure that the set $\cN' = \set{ (n_i)_{i \in I} }{ (n_i)_{i=0}^m \in \cN }$ is infinite.
It follows from the definition of $J$ that all proper sub-sums $\sum_{(i,\tau) \in J'} \tau(a_i) \tau(\beta)^{n_i}$, $\emptyset \neq J' \subsetneq J$, are nonzero for all $(n_i)_{i \in I} \in \cN'$.

From the finiteness of the number of solutions of the $S$-unit equation \cite[Thm.\ 2]{vdPSch91}, we deduce that the solutions $(n_i)_{i \in I}$ of \eqref{eq:83:03} produce up to scaling only a finite number of values of $(\tau(\beta)^{n_i})_{(i,\tau)\in J}$. Thus, applying the pigeon-hole principle and replacing $\cN'$ with an infinite subset, we may assume that $(\tau(\beta)^{n_i})_{(i,\tau)\in J}$ takes the same value, up to scaling, for all $(n_i)_{i\in I} \in \cN'$. It follows that for each $(n_i)_{i\in I} , (n_i')_{i\in I}  \in \cN'$ and for each $(i,\tau),(j,\sigma) \in J$ we have $\tau(\beta)^{n_i'-n_i} = \sigma(\beta)^{n_j'-n_j}$. From Lemma \ref{lem:conjugate-powers} it follows that
\begin{equation}\label{eq:84:01}
{n_i'-n_i} = \pm \bra{n_j'-n_j} \qquad\text{and} \qquad \tau(\beta) = \omega \sigma(\beta)^{\pm 1},
\end{equation}
where $\omega$ is a root of unity that depends on $\tau$ and $\sigma$.

Identifying $K$ with a subfield of $\CC$, we may assume that $\Sigma$ contains the inclusion map $\mathrm{id}$ and that $(i_0,\mathrm{id}) \in J$ for some $i_0$. In particular, taking $(j,\sigma) = (i_0,\mathrm{id})$ we see that $\tau(\beta) = \omega \beta^{\pm 1}$. Partition $J$ as $ J_+ \cup J_-$, where $ J_{\pm}$ is the set of those $(i,\tau) \in J$ for which $\tau(\beta) = \omega \beta^{\pm 1}$. 

We claim that our assumptions guarantee that the set $J_{-}$ is empty. This is immediate if $\beta$ has no conjugates of the form $\omega \beta^{-1}$ with $\omega$ a root of unity. In the second case where $\cN'$ is a subset of $\NN_0^{m+1}$,  we deduce from equation \eqref{eq:84:01} that for all $(n_i)_{i\in I} , (n'_i)_{i\in I}  \in \cN'$ and all $(i,\tau) \in J$ we have 
\begin{align}\label{eq:84:02}
	n'_i =
	\begin{cases}
		n_i + (n'_{i_0}-n_{i_0}) & \text{if } (i,\tau) \in J_+,\\
		n_i - (n'_{i_0}-n_{i_0}) & \text{if } (i,\tau) \in J_-.		
	\end{cases} 
\end{align} If $J_{-}$ were nonempty, this would show that for a given $(n_i)_{i\in I}$ there are only finitely many possibilities for $(n'_i)_{i\in J}$ since $n_i\geq n'_{i_0}-n_{i_0} \geq -n_{i_0}$ for any $i$ such that $(i,\tau)\in J_-$ for some $\tau\in \Sigma$. This would contradict the fact that $\cN'$ is infinite. Thus $J_-$ is empty.

Since $\beta$ satisfies \eqref{eq:pow:non-degen} and since an embedding of $K$ into $\CC$ is uniquely determined by its value on $\b$, the set $J=J_+$ takes the form $J = I \times \{\mathrm{id}\}$. Equation \eqref{eq:83:03} now takes the form \[ \sum_{i \in I} a_i \beta^{n_i}=0,\] which gives the claim. 
\end{proof}

\begin{proposition}\label{prop:x*beta^i}
Let $\beta$ be an algebraic number satisfying \eqref{eq:pow:non-degen}, let $K = \QQ(\beta)$, and let $Y \subset K$ be a finite set. Then there exists an integer $N$ such that for each nonzero $x \in K$, if for each integer $k$ with $0 \leq k< N$ there exists $j(k) \in \NN_0$ and $y(k)\in Y$ such that
\begin{equation}\label{eq:pow:76:1}
	\Tr_{K/\QQ}(\beta^{k}x)=\Tr_{K/\QQ}(\beta^{j(k)}y(k)),
\end{equation}
then there exists an integer $l \in \ZZ$, $z \in Y$, and an automorphism $\sigma$ of $K$ and a root of unity $\omega$ such that $\sigma(\beta) = \omega \beta^{\pm 1}$ and $x = \beta^l \sigma(z)$. 
\end{proposition}
\begin{proof}
Removing if necessary redundant elements of $Y$, we may assume without loss of generality that the ratio $y/y'$ is not a power of $\beta$ for any $y,y' \in Y$.

Let $N$ be a large integer, to be determined in the course of the proof.
	Let $X^m+a_{m-1}X^{m-1}+\cdots+a_0$ be the minimal polynomial of $\beta$ over $\QQ$, and put $a_m=1$. Note that for any $k\geq 0$ we have 
\begin{equation}\label{eq:pow:76:2}	
	\sum_{i=0}^m a_{i} \Tr_{K/\QQ}(\beta^{k+i} x)=0.
\end{equation}
For any nonempty proper subset $I$ of $\{0,\ldots,m\}$, the sequence $(f^{(I)}_k)_{k=0}^{\infty}$ given by
\[ 
	f^{(I)}_k =\sum_{i\in I} a_{i} \Tr_{K/\QQ}(\beta^{k+i}x)
\]
 is a nondegenerate linear recurrence sequence (and is not identically zero by the nondegeneracy of $\Tr_{K/\QQ}$), and so by the Skolem--Mahler--Lech Theorem it has only finitely many zeros. By the results of Schlickewei \cite[Thm.\ 1.1]{Schlickewei-1996}, the number of these zeros is bounded by a constant $C$ that depends only on $K$. Thus, replacing $N$ with $\ip{(N-C)/(C+1)}$ and $x$ with $x' = \beta^{k_0} x$ for suitably chosen $k_0$, we may assume that $f^{(I)}_k \neq 0$ for $0 \leq k < N$. Repeating this procedure for all $I$, we may assume that the sum $\sum_{i\in I} a_{i} \Tr_{K/\QQ}(\beta^{k+i}x)$ is nonzero for each  $0 \leq k < N$ and each nonempty proper subset $I$ of $\{0,\ldots,m\}$. 
 
  \newcommand{\cJ}{\mathcal{J}}
 Combining \eqref{eq:pow:76:1}, \eqref{eq:pow:76:2} and the reduction above, for $0 \leq k < N-m$ we have
 \begin{equation}\label{eq:pow:76:3}	
	\sum_{i=0}^m a_{i} \Tr_{K/\QQ}(\beta^{j(k+i)}y(k+i))=0,
\end{equation}
and no proper nonempty sub-sum of \eqref{eq:pow:76:3} vanishes.
 Applying Proposition \ref{prop:lemmasumsoftraces} with $a_i' = a_i y_i$ for all possible choices of $(y_i)_{i=0}^m \in Y^{m+1}$,  we conclude that there exists a finite set $\cJ \subset \NN_0^{m+1}$ such that for each $0 \leq k < N-m$ we have
 \begin{equation}\label{eqn:powbetaj}
 \sum_{i=0}^m a_i \beta^{j(k+i)}y(k+i)  =0
 \end{equation}
 unless $(j(k+i))_{i=0}^m \in \cJ$. Furthermore, no proper nonempty sub-sum of \eqref{eqn:powbetaj} vanishes. 
 
As a consequence of Lemma \ref{lem:pow:equal-traces}\ref{it:pow:53:A}, for each $(j_i)_{i=0}^m \in \NN_0^{m+1}$ and $(y_i)_{i=0}^m \in Y^{m+1}$, there exists at most one value of $k$ with $0 \leq k < N-m$ such that $(j(k+i))_{i=0}^m = (j_i)_{i=0}^m$ and $(y(k+i))_{i=0}^m = (y_i)_{i=0}^m$. Replacing once more $N$ with $\ip{(N-\abs{\cJ}\cdot \abs{Y}^{m+1})/(\abs{\cJ}\cdot \abs{Y}^{m+1}+1)}$, we may assume that \eqref{eqn:powbetaj} holds for all $0 \leq k < N-m$.
  	
It follows from the finiteness of the number of solutions of $S$-unit equations that the number of $(m+1)$-tuples $(\beta^{j(k+i)}y(k+i))_{i=0}^m$ that satisfy \eqref{eqn:powbetaj}, regarded up to scaling, is bounded by a constant $C'$ that depends only on $\beta$ and $Y$. Of course, the number of $(m+1)$-tuples $(y(k+i))_{i=0}^m$ is bounded by $\abs{Y}^{m+1}$. Letting $C'' = C' \abs{Y}^{m+1}$ and assuming, as we may, that $N > C'' + m$, we can find $k_1,k_2$ with $0 \leq k_1 < k_2 \leq C''$ such that 
\begin{equation}\label{eq:pow:76:5}	
  	  j(k_1+i)-j(k_1) = j(k_2+i)-j(k_2) \qquad \text{and}\qquad 
  	  y(k_1+i) = y(k_2+i)
\end{equation}
for all $i$ with $0 \leq i \leq m$. It follows from \eqref{eqn:powbetaj}  that, if for some $k$ we are given $j(k+1)-j(k),\ldots,j(k+m-1)-j(k)$ and $y(k),y(k+1),\dots,y(k+m-1)$, then we can uniquely determine the value of $\beta^{j(k+m)-j(k)}y(k+m)$, and hence also the values of  $j(k+m)-j(k)$ and $y(k+m)$. As a consequence, \eqref{eq:pow:76:5} holds, more generally, for all $i$ with $0 \leq i < N-k_2$. Setting $d := k_2 - k_1$ and $e := j(k_2) - j(k_1)$, for all $k$ with $k_1 \leq k< N-d$ we have
\begin{equation}\label{eq:pow:76:6}	
  	  j(k+d) = j(k) + e\qquad \text{and}\qquad y(k+d) = y(k).
\end{equation}  
Iterating \eqref{eq:pow:76:6}, we conclude that for any $k$ and $l\geq 0$ we have
\begin{equation}\label{eq:pow:76:7}	
  	  j(k+ld)= j(k) + le,
\end{equation}    
provided that $k_1 \leq k < N - ld$. Recalling how $j(k)$ was defined, we conclude from \eqref{eq:pow:76:7} that
\begin{equation}\label{eq:pow:76:8}	
	\Tr_{K/\QQ}(\beta^{k+ld} x)=\Tr_{K/\QQ}(\beta^{j(k)+le}y(k)).
\end{equation}    

Let $\mu$ be the order of the (cyclic) group of roots of unity contained in $K$. Choose $k$ to be an integer such that $k \geq k_1$ and $k$ is divisible by $\mu$. Assume moreover, as we may, that $N$ is sufficiently large so that $k+(2m-1)d  < N$. We conclude from \eqref{eq:pow:76:8} and Lemma  \ref{lem:pow:equal-traces}\ref{it:pow:53:B} that there exists an automorphism $\sigma$ of $K$ with $\sigma(\beta^d) = \beta^e$ and $\sigma(\beta^k x) = \beta^{j(k)} y(k)$. By Lemma \ref{lem:conjugate-powers} we have $e = \pm d$ and $\sigma(\beta) = \omega \beta^{\pm 1}$ for some root of unity $\omega$. Since $k$ is divisible by $\mu$, we have $\sigma(\beta^k)=\beta^{\pm k}$, and so $\sigma(x)=\beta^{j(k)\mp k} y(k)$. This concludes the proof. \end{proof}

We have now all the technical tools necessary for the proof of  Theorem \ref{thm:Tr-vs-Pownew}.
\begin{proof}[Proof of Theorem \ref{thm:Tr-vs-Pownew}]
We may assume that $\beta \neq 0$. Put $K = \QQ(\beta)$ and let $X$ denote the set of values of the sequence $(n_i)_{i=0}^{\infty}$.  If it were the case that $\abs{\tau(\beta)} \leq 1$ for all embeddings $\tau$ of $K$ into $\CC$, then $X$ would be a bounded subset of $\QQ$. Since a generalised polynomial map on a bounded real interval has only a finite number of discontinuities, $X$ would then be finite, and $\beta$ would be a root of unity, in which case the claim would be clear. Thus, we may assume that there exists an embedding $\tau_0$ of $K$ into $\CC$ such that $\abs{\tau_0(\beta)} > 1$.

Let $d\in \NN$ be the lowest common multiple of the orders of all roots of unity that may occur as quotients or products of two conjugates of $\beta$. Then, $\gamma = \beta^d$ satisfies \eqref{eq:pow:non-degen} and has no conjugates of the form $\omega \gamma^{-1}$ with $\omega$ a root of unity except possibly for $\omega = 1$.   Put $M = \QQ(\gamma)$

Using Lemma \ref{lem:lrseqform}\ref{lem:lrseqform1}, we choose $z\in K$ so that $$n_i = \Tr_{K/\QQ}(\beta^i z) \quad \text{for all } i\in \NN_0.$$ Let $S =\set{\Tr_{K/M}(\beta^r z)}{0 \leq r < d}$. Using the transitivity of the trace, we may write each $\Tr_{K/\QQ}(\beta^{i}z)$ in the form $$\Tr_{K/\QQ}(\beta^{i}z)=\Tr_{M/\QQ}(\gamma^{j} s) \quad \text{for } j=\lfloor i/d\rfloor,\ s=\Tr_{K/M}(\beta^{i-dj}z).$$ Thus, the set $X$ takes the form $$X=\set{\Tr_{M/\QQ}(\gamma^i s)}{i\in \NN_0,\ s\in S}.$$  Observe that $S$ contains some nonzero element, since otherwise $(n_i)_{i=0}^{\infty}$ would be identically zero.

\begin{claim*} The set $$\{\gamma^{l} s \mid l\in \NN_0,\ s\in S\}$$ is a generalised polynomial subset of $M$.
\end{claim*}
\begin{claimproof}
Let $N$ be a large integer, to be determined shortly, and consider the set $$A=\{ x\in M \mid \Tr_{M/\QQ}(\gamma^k x) \in X \text{ for all } 0\leq k < N\}.$$ Since $X$ is a generalised polynomial subset of $\QQ$, and since the family of generalised polynomial subsets of $M$ is closed under finite intersections and taking preimage by a generalised polynomial map, we see that $A$ is a generalised polynomial subset of $M$. 

It is clear that $$\set{\gamma^l s }{ l \in \NN_0,\ s \in S} \subset A.$$ 
In the opposite direction, observe that if $x$ belongs to $A$, then for every $0\leq k <N$ there exists some $j(k)\in \NN_0$ and $s(k)\in S$ such that $$\Tr_{M/\QQ}(\gamma^k x) = \Tr_{M/\QQ}(\gamma^{j(k)}s(k)).$$
Choosing $N$ to be sufficiently large for the claim of Proposition \ref{prop:x*beta^i} to hold, we infer that $x = \gamma^l \sigma(s)$ for some $s \in S$, $l \in \ZZ$, and automorphism $\sigma$ of $K$ with $\sigma(\gamma) = \gamma^{\pm 1}$.   We will consider two cases depending whether or not $\gamma$ and $\gamma^{-1}$ are conjugate.

{\it Case I:} (Suppose that $\gamma$ is not conjugate to $\gamma^{-1}$). The above reasoning shows that  $$\set{\gamma^l s }{ l \in \NN_0,\ s \in S} \subset A \subset \set{\gamma^l s }{ l \in \ZZ,\ s \in S}.$$ Suppose that $x=\gamma^l s$ belongs to $A$ for some $l\in \ZZ$, $l <0$, $s\in S$. Then in particular $$\Tr_{M/\QQ}(\gamma^l s) = \Tr_{M/\QQ}(\gamma^j s') \quad \text{for some } j\in\NN_0,\ s'\in S.$$ From the second part of Proposition \ref{prop:lemmasumsoftraces} we deduce that there are only finitely many possibilities for $x$. Indeed, except for finitely many possible values of $l$, we have one of three possibilities: $\gamma^l s=0$, in which case $x = 0$; $\gamma^l s = \gamma^j s' \neq 0$, in which case $\gamma^{\abs{l}} \leq s/s'$ and hence there are only finitely many possibilities for $x$; or $\gamma^j s'=0$, in which case $\Tr_{M/\QQ}(\gamma^l s)=0$, and another application of Proposition \ref{prop:lemmasumsoftraces} shows that there are only finitely many possibilities for $x$. Thus the set $\set{\gamma^l s }{ l \in \NN_0,\ s \in S}$ is generalised polynomial since it differs from $A$ only on a finite subset. 

{\it Case II:} (Suppose that $\gamma$ is conjugate to $\gamma^{-1}$). Let $\sigma$ be the automorphism of $M$ such that $\sigma(\gamma)=\gamma^{-1}$. Since $$\Tr_{M/\QQ}(\gamma^l s)=\Tr_{M/\QQ}(\sigma(\gamma^{l}s))=\Tr_{M/\QQ}(\gamma^{-l}\sigma(s)),$$ we deduce that $$\set{\gamma^l s }{ l \in \NN_0,\ s \in S} \cup \set{\gamma^{-N+1-l} \sigma(s) }{ l \in \NN_0,\ s \in S}  \subset A \subset \set{\gamma^l s }{ l \in \ZZ,\ s \in S\cup\sigma(S)}.$$ Suppose that $x=\gamma^l s$ belongs to $A$ for some $l\in \ZZ$, $l <0$, $s\in S$. Then in particular $$\Tr_{M/\QQ}(\gamma^l s) = \Tr_{M/\QQ}(\gamma^j s') \quad \text{for some } j\in\NN_0,\ s'\in S.$$ On the other hand, we have $$\Tr_{M/\QQ}(\gamma^l s) = \Tr_{M/\QQ}(\sigma(\gamma^l s))=\Tr_{M/\QQ}(\gamma^{-l} \sigma(s)).$$ Applying the first part of Proposition \ref{prop:lemmasumsoftraces} to the equality $$\Tr_{M/\QQ}(\gamma^{-l} \sigma(s)) = \Tr_{M/\QQ}(\gamma^j s'),$$ we deduce as before that except for finitely many possible values of $x$ we have $\sigma(x)=\gamma^{-l} \sigma(s)=\gamma^j s'$, so $x=\gamma^{-j}\sigma(s')$. A similar (simpler) reasoning shows that if $x=\gamma^l \sigma(s)$ belongs to $A$ for some $l\in \NN_0$, $s\in S$, then except for finitely many possible values of $x$ we have $x=  \gamma^j s'$ for some $j\in \NN_0$, $s'\in S$. We conclude that the set $$A'=\set{\gamma^l s }{ l \in \NN_0,\ s \in S} \cup \set{\gamma^{-l} \sigma(s)}{l\in \NN_0, s\in S}$$  is generalised polynomial since it differs from $A$ only on a finite subset. Let $C$ be the generalised polynomial subset of $M$ consisting of the elements $x\in M$ such that $\abs{\tau_0(x)}<1$. Removing $C$ from $A'$ retains all the elements $x=\gamma^l s$ with $s\in S$, $s\neq 0$, and $l\in \NN_0$ sufficiently large, and eliminates $x=0$ and all the elements $x=\gamma^{-l} \sigma(s)$ with $s\in S$, and $l$ sufficienly large. Thus, the set $\set{\gamma^l s }{ l \in \NN_0,\ s \in S}$ is generalised polynomial since it differs from $A'\setminus C$ only on a finite subset. 
\end{claimproof}

Let $S_0$ be the set obtained from $S$ by performing the following operations: if $0\in S$, replace $S$ by $S\setminus \{0\}$; replace $S$ by $s^{-1}S$ for some $s\in S$; remove any $s \in S$ such that $s=\gamma^j s'$ for some $j\in \NN$, $s'\in S$. We obtain in this manner a finite set $S_0$ with $0\notin S_0$, $1\in S_0$, and such that $D_0=\set{\gamma^l s}{l\in \NN_0,\ s\in S_0}$ is a generalised polynomial subset of $M$. Consider the set $$D_1 = \{x\in M \mid s x \in D_0 \text{ for all }s\in S_0\}.$$ By construction, $D_1$ is a generalised polynomial subset of $M$ and can be written in the form $$D_1 =  \{\gamma^l s \mid l\in \NN_0,\ s\in S_1\} $$ for some finite set $S_1$ such that $0\notin S_1$, $1\in S_1$, and the quotient of two different elements of $S_1$ is not an integral power of $\gamma$. Moreover, we have $|S_1|\leq |S_0|$, with equality occuring only if for any $s,s'\in S_0$, the product $s s'$ is equal to $\gamma^t s''$ for some $t\in \ZZ$ and $s''\in S_0$. This latter condition is equivalent to the condition that the image of $S_0$ in the quotient group $M^*/\langle \gamma\rangle$ is closed under multiplication, and hence is a finite group.

Continuing this procedure, we obtain a sequence of finite sets $(S_n)_{n=0}^{\infty}$ with $|S_0|\geq |S_1|\geq \cdots$ and corresponding generalised polynomial sets  $D_0\supset D_1\supset \cdots$. Let $m$ be such that $|S_m|=|S_{m+1}|$. This means that $S_m$  is such that $0\notin S_m$, $1\in S_m$, the quotient of two different elements of $S_m$ is not an integral power of $\gamma$, and the image of $S_m$ in  $M^*/\langle \gamma\rangle$ is a finite group. It follows that any $s\in S_m$ is a rational power of $\gamma$. Let $G$ be the subgroup of $M^*$ generated by $S_m$ and $\gamma$. Then $G$ is finitely generated and $\langle \gamma \rangle$ is its subgroup of finite index. Hence, by Chevalley's theorem  \cite[Thm.\ 1]{Chevalley51}, $\langle \gamma\rangle$ is a congruence subgroup of $G$, meaning that $\langle \gamma\rangle=G \cap \Lambda$ for some $\Lambda\subset M$ that is a union of finitely many cosets of a lattice. Thus, the set $$D_m\cap \Lambda = \{\gamma^l s \mid l\in \NN_0, s\in S_m\} \cap \Lambda = \{\gamma^l \mid l\in \NN_0\}$$ is a generalised polynomial subset of $M$. Since generalised polynomial sets are closed under dilations and finite unions, this also implies that $\{\beta^l \mid l\in \NN_0\}$ is a generalised polynomial subset of $M$, and hence by Proposition \ref{prop:gpsetsexam}\ref{prop:gpsetsexam1} also of $K$.
\end{proof}

\section{Non-hereditary generalised polynomial setshere}\label{sec:No-Heir}

In light of Theorem \ref{thm:Pisot-hereditary-Z}, it is natural to ask which generalised polynomial sets of integers are hereditary. It is not hard to see that each set $E \subset \ZZ$ with positive density
\begin{equation}\label{eq:her:75:1}
	d(E) := \lim_{N \to \infty} \frac{\abs{E \cap [-N,N]}}{2N+1} > 0
\end{equation}
has a subset $E'$ that does not have density, meaning that
\[
	\liminf_{N \to \infty} \frac{\abs{E \cap [-N,N]}}{2N+1}  =: \underline{d}(E) < \overline{d}(E) := \limsup_{N \to \infty} \frac{\abs{E \cap [-N,N]}}{2N+1}.
\]
Since the density exists for each generalised polynomial set, no generalised polynomial set with positive density is hereditary.
When it comes to sets with zero density, it remains the case that we expect most of them to not be hereditary, but proving this becomes more difficult. However, we can at least show that not all of them are hereditary.

\begin{theorem}\label{thm:non-hereditarynew}
	There exists a generalised polynomial set $E \subset \ZZ$ with $d(E) = 0$ as well as a subset $E' \subset E$ that is not a generalised polynomial set.
\end{theorem}

Our proof of Theorem \ref{thm:non-hereditary} relies on two components from \cite{AdamczewskiKonieczny}. The first ingredient is a polynomial bound on subword complexity of finitely-valued generalised polynomials. Recall that the \emph{subword complexity} $p_a$ of a sequence $a \colon \ZZ \to \Sigma$ taking values in a finite alphabet $\Sigma$ is the map that assigns to a positive integer $N$ the number of distinct length-$N$ subsequences of $a$:
\begin{equation}\label{eq:Non:A:sub-comp}
	p_a(N) := \abs{ \set{ \bra{a(m+n)}_{n=0}^{N-1} }{m \in \ZZ} }.
\end{equation}
The subword complexity of a sequence $\NN \to \Sigma$ is defined analogously. If $\abs{\Sigma} = k$, we have the trivial upper bound $p_a(N) \leq k^N$. 

\begin{theorem}[{\cite[Thm.\ A]{AdamczewskiKonieczny}}]\label{thm:subword-complexity}
	Let $g \colon \ZZ \to \Sigma$ be a generalised polynomial taking values in a finite set $\Sigma \subset \RR$. Then there exists a constant $C = C(g) > 0$ such that $p_g(N) = O(N^C)$ as $N \to \infty$.
\end{theorem}

The second ingredient is the existence of generalised polynomial sets $E$ with zero density, but with the expression in \eqref{eq:her:75:1} converging to $0$ arbitrarily slowly. The following result was originally stated for subsets of $\NN$, but the adaptation to $\ZZ$ is immediate.
\begin{theorem}[{\cite[Prop.\  8.12]{AdamczewskiKonieczny}}]\label{thm:slow-decay}
	Let $f \colon \NN \to [0,1]$ be a sequence with $f(N) \to 0$ as $N \to \infty$. Then there exists a generalised polynomial set $E \subset \ZZ$ such that $d(E) = 0$ and $\abs{E \cap [0,N)} \geq f(N)N$.
\end{theorem}

The following consequence of Theorem \ref{thm:slow-decay}, juxtaposed with Theorem \ref{thm:subword-complexity}, will almost immediately yield a proof of Theorem \ref{thm:non-hereditary}.

\begin{proposition}\label{prop:Non:A:p-is-large}
	Let $h \colon \NN \to [0,1]$ be a sequence with $h(L) \to 0$ as $L \to \infty$. There exists a generalised polynomial set $E \subset \ZZ$ with $d(E) = 0$ and a subset $F \subset E$ such that $p_{1_F}(L) \geq 2^{h(L)L}$ for all $L \in \NN$.
\end{proposition}
\begin{proof}
	Replacing $h(L)$ with $\ceil{h(L)L}/L$, we may freely assume that $h(L)L$ is an integer for all $L \in \NN$.
	Let $M_L, N_L$ be sequences of integers satisfying $N_0 = M_0 = 0$ and, for $L \geq 1$,
	\begin{equation}\label{eq:Non:A:N-def}
		N_{L} := N_{L-1} + L \cdot M_L, \qquad M_L \geq \frac{N_{L-1}+ 2^{(1+h(L))L}L }{(h(L)-h(L)^2)L}.
	\end{equation}
		
		Let $f \colon \NN \to [0,1]$ be a sequence satisfying 
\begin{equation}\label{eq:Non:A:f-ass}
 	f(N_L) \geq 2h(L) - h(L)^2 \qquad \text{and} \qquad f(N) \to 0 \text{ as } N \to \infty.
\end{equation}
Let $E$ be a set whose existence is asserted in Theorem \ref{thm:slow-decay}. We will construct $F$ as the intersection of a descending sequence of sets $E_L$, where $E_0 := E$. Pick a positive integer $L$. We can decompose the interval $[0,N_L)$ as
	\[ [0,N_L) = [0,N_{L-1}) \cup \bigcup_{m=0}^{M_L-1} [N_{L-1}+ mL,N_{L-1}+(m+1)L). \]
	Let $M_L^+$ denote the number of integers $m \in [0, M_L)$ such that 
	\[
		\abs{ E \cap [N_{L-1}+ mL,N_{L-1}+(m+1)L)} \geq h(L)L,
	\]
	Then 
	\[
		f(N_L)N_L \leq \abs{ E \cap [0,N_L)} \leq N_{L-1} + M_{L}^+ L + M_{L} h(L)L.
	\]
	Combined with \eqref{eq:Non:A:N-def} and \eqref{eq:Non:A:f-ass}, this implies that
	\[
		M_{L}^+ \geq 2^{(1+h(L))L}.
	\]
	Put $H:= 2^{h(L)L}$. Applying the pigeonhole principle, we conclude that there exists a set $A \subset [0,L)$ with $\abs{A} \geq h(L)L$ and positions $0 \leq m_0 < m_1 < \dots < m_{H-1} < M_{L}$ such that 
	\[
		E \cap [N_{L-1}+ m_j L,N_{L-1}+(m_j+1)L) = N_{L-1} + m_j L + A \quad \text{ for all } 0 \leq j < H.
	\]	
	Let $A_0,A_1,\dots, A_{H-1}$ be $H$ different subsets of $A$. Put $B_j := A \setminus A_j$ and
	\[
		E_L := E_{L-1} \setminus 
		\bigcup_{j=0}^{H-1} \bra{  N_{L-1} + m_jL + A_j}.
	\]
	In particular, for each $0 \leq j < H$ we have
	\[
	\bra{ E_{L} - (N_{L-1} + m_jL) } \cap [0,L) = B_j.
	\]
	In particular, $1_{B_j}$ is a subsequence of $1_{E_L}$, and thus the indicator function of $E_L$ has at least $2^{h(L)L}$ length-$L$ subsequences, all of which appear at positions between $N_{L-1}$ and $N_L$.
	We put $F := \bigcap_{L =0}^\infty E_L$. It follows from the construction above that the subword complexity of $F$ satisfies $p_{1_F}(L) \geq 2^{h(L)L}$ for all $L \in \NN$.
\end{proof}

\begin{proof}[Proof of Theorem \ref{thm:non-hereditary}]
Let $h(L) := 1/\sqrt{L}$, and let $E,F$ be some sets satisfying the claim of  Proposition \ref{prop:Non:A:p-is-large}. Then $F$ is not a generalised polynomial set by Theorem \ref{thm:subword-complexity}.
\end{proof}
\newcommand{\etalchar}[1]{$^{#1}$}

\end{document}